\documentclass[11pt,reqno]{amsart}

\usepackage{amsmath,amssymb,amsthm,mathtools,mathrsfs}
\usepackage{enumitem}
\usepackage{microtype}
\usepackage[colorlinks=true,
            linkcolor=blue,
            citecolor=blue,
            urlcolor=blue]{hyperref}

\numberwithin{equation}{section}

\newtheorem{theorem}{Theorem}[section]

\newcommand{\N}{\mathbb{N}}
\newcommand{\one}{\mathbf{1}}
\newcommand{\dd}{\,d}
\newcommand{\vol}{\operatorname{Vol}}

\title[Synchronized sign-changing Liouville systems]
{Synchronized Solutions for Liouville Systems\\
with Sign-Changing Coefficients}

\author{Woongbae Park}
\address{Department of Mathematics, University of Rochester,
Rochester, NY 14627, USA}
\email{wpark14@ur.rochester.edu}

\author{Lei Zhang}
\address{Department of Mathematics, University of Florida,
Gainesville, FL 32611-8105, USA}
\email{leizhang@ufl.edu}
\thanks{L.~Zhang is partially supported by Simons Foundation Grant
SFI-MPS-TSM-00013752.}


\begin{document}

\begin{abstract}
We consider singular Liouville systems on closed Riemann surfaces with a
common sign-changing coefficient.  Direct continuation from a positive
coefficient is obstructed by the unavoidable appearance of degenerate zeroes
along such a homotopy.  We instead identify an intrinsic ray in the parameter
space on which the system admits synchronized solutions and reduces exactly to
a scalar sign-changing mean-field equation.  Combining this reduction with the
variational existence theory of De Marchis--L\'opez-Soriano--Ruiz and the nodal
singularity compactness theorem of Wu, we obtain existence for every
noncritical parameter on this ray.  Positive singular sources are allowed on
the nodal curve.  In particular, when the surface has nonpositive Euler
characteristic and every negative component of the coefficient is a disk, the
topological hypothesis required by the scalar min--max construction is
automatic.
\end{abstract}

\maketitle

\section{Introduction}

Let $(M,g)$ be a connected closed Riemann surface normalized by
$\vol_g(M)=1$.  We study the system
\begin{equation}\label{eq:system-intro}
 \Delta_g u_i+
 \sum_{j=1}^n a_{ij}\rho_j
 \left(
   \frac{h e^{u_j}}{\displaystyle\int_M h e^{u_j}\dd V_g}-1
 \right)
 =\sum_{\ell=1}^L4\pi\alpha_\ell(\delta_{p_\ell}-1),
 \qquad i=1,\ldots,n,
\end{equation}
where $A=(a_{ij})$ is an invertible coupling matrix,
$\rho_j>0$, and the singular strengths satisfy $\alpha_\ell\geq0$.
The coefficient $h$ is allowed to change sign, and solutions are sought in the
natural branch
\begin{equation}\label{eq:positive-denominator}
 \int_M h e^{u_i}\dd V_g>0,
 \qquad i=1,\ldots,n.
\end{equation}
The equation is invariant under the addition of an arbitrary constant to each
component.  Throughout, multiplicity is understood modulo these additive
constants, or equivalently after imposing one normalization on each component.

For positive coefficients, compactness, degree theory, and blow-up profiles of
Liouville systems have been studied extensively; see, for example,
\cite{LinZhangProfile,LinZhangDegree}.  Sign-changing coefficients introduce a
different obstruction.  The normalization in \eqref{eq:positive-denominator}
controls a signed rather than a positive measure, and the nodal set
\[
 \Gamma=\{x\in M:h(x)=0\}
\]
becomes a possible transition region.  Even in the scalar equation, the
compactness analysis requires integral estimates on the negative side and a
moving-plane comparison across $\Gamma$.

For the scalar singular mean-field equation, De Marchis,
L\'opez-Soriano, and Ruiz \cite{DMLR} established compactness, existence, and
generic multiplicity results under the regular-nodal-set hypothesis and the
condition that prescribed singularities avoid the nodal curve.  Wu
\cite{Wu} subsequently proved that the latter restriction can be removed when
the singular strengths are nonnegative: positive singular sources may lie
directly on $\Gamma$.  These existence results are intrinsic to the
sign-changing problem.  They are obtained from the topology of low energy
sublevels and do not deform the coefficient to a positive function.

The purpose of this paper is to transfer that intrinsic mechanism to a natural
and nontrivial class of Liouville systems.  Let
\[
 A^{-1}=(a^{ij}),
 \qquad
 q_i=\sum_{j=1}^n a^{ij},
 \qquad
 q=A^{-1}\one.
\]
If $q_i>0$ and
\begin{equation}\label{eq:synchronized-ray-intro}
 \rho_i=\lambda q_i,
\end{equation}
then
\[
 A\rho=\lambda\one.
\]
Consequently, the diagonal ansatz
$u_1=\cdots=u_n$ reduces every equation of the system to the same scalar
equation with parameter $\lambda$.  The resulting ray
\[
 \{\lambda A^{-1}\one:\lambda>0\}
\]
is determined entirely by the coupling matrix and will be called the
\emph{synchronized ray}.

Our main theorem gives a solution of \eqref{eq:system-intro} at every
point of this ray which is with scalar parameter $\lambda$ outside the critical parameter space $\Lambda_h$.
The topological assumption on the positive
region needed in \cite{DMLR} is automatic under the geometric hypotheses used
here: $\chi(M)\leq0$ and every component of $\{h<0\}$ is a disk.  The theorem
also permits positive singular sources on the nodal curve by invoking the
compactness result of \cite{Wu}.

The restriction to the synchronized ray is substantive.  For symmetric coupling matrices, the formal system
functional contains the quadratic form of $A^{-1}$.  Under the usual
strong-interaction hypothesis this quadratic form is indefinite, so the scalar
barycenter description of low sublevels does not directly extend to arbitrary
parameter vectors.  We therefore make no claim here for general $\rho$.
%and do
%not use a positive-to-sign-changing coefficient homotopy.

\section{Setting and scalar input}

We impose the following assumptions on the coefficient:
\begin{equation}\label{eq:h-hyp}
 h\in C^{2,\vartheta}(M),\qquad
 h\ \text{changes sign},\qquad
 |\nabla_g h(x)|>0\quad\text{for every }x\in\Gamma,
\end{equation}
for some $\vartheta\in(0,1)$.  Write
\[
 \Gamma^+=\{h>0\},\qquad
 \Gamma^-=\{h<0\},\qquad
 \Gamma=\{h=0\}.
\]
We assume that every connected component of $\Gamma^-$ is diffeomorphic to a
disk.  The regularity condition in \eqref{eq:h-hyp} implies that $\Gamma$ is a
finite union of embedded $C^{2,\vartheta}$ closed curves.

The prescribed points $p_1,\ldots,p_L$ are distinct and
\begin{equation}\label{eq:source-hyp}
 \alpha_\ell\geq0,
 \quad
 \text{ and if } p_\ell\in\Gamma , \, \alpha_\ell>0.
\end{equation}
We set
\begin{equation}\label{eq:Iplus}
 \mathcal I_+(h)=\{\ell\in\{1,\ldots,L\}:h(p_\ell)>0\}
\end{equation}
and define the scalar critical set
\begin{equation}\label{eq:Lambda}
 \Lambda_h=
 \left\{
  8\pi m+8\pi\sum_{\ell\in B}(1+\alpha_\ell):
  m\in\N\cup\{0\},\ B\subseteq\mathcal I_+(h)
 \right\}\setminus\{0\}.
\end{equation}
A marked point with $\alpha_\ell=0$ can be deleted without changing the
equation or $\Lambda_h$: its contribution $8\pi$ in \eqref{eq:Lambda}
can be absorbed into $8\pi m$.  Thus, when applying scalar existence
theorems stated for positive singular strengths, all remaining strengths
may be taken to be strictly positive.

Only singular points in $\Gamma^+$ occur in \eqref{eq:Lambda}.  Blow-up at a
singular point in $\Gamma^-$ is excluded by the negative-region estimate, and
blow-up at a positive singular source on $\Gamma$ is excluded by the nodal
moving-plane argument of \cite{Wu}.

We recall briefly how the scalar equation fits the variational framework.  Let
$G(x,p)$ be normalized by
\begin{equation}\label{eq:green}
 \Delta_gG(x,p)=-\delta_p+1,
 \qquad
 \int_MG(x,p)\dd V_g(x)=0.
\end{equation}
For a scalar solution $u$, set
\begin{equation}\label{eq:singular-removal}
 U=u+4\pi\sum_{\ell=1}^L\alpha_\ell G(x,p_\ell),
 \qquad
 H=h\exp\left(-4\pi\sum_{\ell=1}^L
                  \alpha_\ell G(x,p_\ell)\right).
\end{equation}
In a local isothermal coordinate $z$ centered at $p_\ell$,
\begin{equation}\label{eq:local-weight}
 H(z)=h(z)|z|^{2\alpha_\ell}e^{\psi_\ell(z)},
\end{equation}
where $\psi_\ell$ is smooth.  In particular, $H$ is bounded locally because
$\alpha_\ell\geq0$.  The regularized variable $U$ solves
\begin{equation}\label{eq:regularized-scalar}
 \Delta_g U+\lambda\left(
 \frac{He^U}{\displaystyle\int_MHe^U\dd V_g}-1\right)=0.
\end{equation}
The positive-integral branch admits the unique normalized representative
$v=U-\log\int_MHe^U\dd V_g$, which solves
\begin{equation}\label{eq:normalized-scalar}
 \Delta_g v+\lambda(He^v-1)=0,
 \qquad
 \int_MHe^v\dd V_g=1.
\end{equation}
For the variational formulation we instead use the mean-zero domain
\begin{equation}\label{eq:scalar-domain}
 \mathcal X_H=\left\{w\in H^1(M):
 \int_Mw\dd V_g=0,\quad\int_MHe^w\dd V_g>0\right\}
\end{equation}
and the functional
\begin{equation}\label{eq:scalar-functional}
 \mathcal J_\lambda(w)
 =\frac12\int_M|\nabla_gw|^2\dd V_g
  -\lambda\log\int_MHe^w\dd V_g,
 \qquad w\in\mathcal X_H.
\end{equation}
The domain is nonempty: one can choose a smooth function whose exponential
is sufficiently large on a ball compactly contained in
$\Gamma^+\setminus\{p_1,\ldots,p_L\}$ and then subtract its mean.
A critical point $w$ solves \eqref{eq:regularized-scalar} with $U=w$.
It yields a solution of \eqref{eq:normalized-scalar} only after the shift
\begin{equation}\label{eq:normalization-shift}
 v=w-\log\int_MHe^w\dd V_g.
\end{equation}
Thus the mean-zero and positive-integral normalizations are used separately.
The corresponding singular solution is understood distributionally, with
$u+4\pi\sum_\ell\alpha_\ell G(\cdot,p_\ell)\in H^1(M)$.

For $0<\lambda<8\pi$, the Moser--Trudinger inequality makes
\eqref{eq:scalar-functional} coercive on its natural domain and gives a
minimizer.  If
\begin{equation}\label{eq:supercritical-interval}
 8\pi k<\lambda<8\pi(k+1),\qquad k\geq1,
%\lambda \in (8\pi n_k, 8\pi n_{k+1})
\end{equation}
%where $\Lambda_h = \{n_k : k = 1,2, \ldots\}$,
the functional is unbounded below.  The main variational result of
\cite{DMLR} gives a critical point when either $\Gamma^+$ has more than $k$
components or one component of $\Gamma^+$ is not simply connected, provided
$\lambda\notin\Lambda_h$; see \cite[Theorem~2.2]{DMLR}.
Wu's nodal a priori estimate \cite{Wu} extends the scalar compactness input
to positive singular sources on $\Gamma$.  Combined with the variational
construction of \cite{DMLR}, this gives the same conclusion under
\eqref{eq:source-hyp}.  The nodal estimate is applied to the original
coefficient $h$ together with its singular factors, not by assuming that the
regularized weight $H$ has a nonvanishing gradient on its zero set.

\section{The main result}

We assume that $A=(a_{ij})$ satisfies
\begin{enumerate}[label=\textnormal{(A\arabic*)},leftmargin=2.6em]
 \item\label{A1} $A$ is symmetric, nonnegative, irreducible, and invertible;
 \item\label{A2} if $A^{-1}=(a^{ij})$, then
 \[
   a^{ii}\leq0,
   \qquad a^{ij}\geq0\quad(i\neq j),
   \qquad \sum_{j=1}^na^{ij}\geq0;
 \]
 \item\label{A3} the row sums of $A^{-1}$ are strictly positive:
 \[
   q_i:=\sum_{j=1}^na^{ij}>0,
   \qquad i=1,\ldots,n.
 \]
\end{enumerate}
The strictness in \ref{A3} is exactly what is needed to place the synchronized
ray inside $(0,\infty)^n$.

\begin{theorem}[Intrinsic synchronized existence]\label{thm:main}
Let $(M,g)$ be a connected closed Riemann surface with $\vol_g(M)=1$ and
$\chi(M)\leq0$.  Suppose that $h$ satisfies \eqref{eq:h-hyp}, every connected
component of $\Gamma^-$ is diffeomorphic to a disk, and the singular data
satisfy \eqref{eq:source-hyp}.  Let $A$ satisfy \ref{A1}--\ref{A3}, put
\begin{equation}\label{eq:q-def}
 q=A^{-1}\one,
\end{equation}
and, for $\lambda>0$, define
\begin{equation}\label{eq:rho-def}
 \rho=\lambda q.
\end{equation}
Then, for every
\begin{equation}\label{eq:lambda-noncritical}
 \lambda\in(0,\infty)\setminus\Lambda_h,
\end{equation}
system \eqref{eq:system-intro} has a synchronized solution
\begin{equation}\label{eq:synchronized-solution}
 u_1=\cdots=u_n=u,
 \qquad
 \int_Mhe^u\dd V_g>0.
\end{equation}
In particular, positive singular sources are allowed on the nodal curve
$\Gamma$.  Moreover, the diagonal map induces an injection from scalar
solutions modulo additive constants into system solutions modulo independent
additive constants.  Consequently, any scalar multiplicity lower bound
applicable to the given data is also a lower bound for synchronized system
solutions.
\end{theorem}

The existence and lifting conclusions in fact require only that $A$ be
invertible and $A^{-1}\one\in(0,\infty)^n$.  Assumptions
\ref{A1}--\ref{A3} place the result in the customary strongly interacting
class and are retained for the discussion in Section~\ref{sec:scope}.

The multiplicity assertion is conditional on the hypotheses of the scalar
result being used.  In particular, generic bounds from \cite{DMLR,Wu}
transfer on the parameter spaces where the corresponding scalar genericity
statement is available.  We do not assert genericity within a separately
constrained class requiring specified points to remain on $\Gamma$.

Several features of Theorem~\ref{thm:main} are worth emphasizing.  First, the
coefficient remains fixed throughout the proof; no deformation from $h$ to a
positive function is used. 
In fact, this deformation yields unavoidable degenerate zeros which make it hard to do continuation method.
Second, the conclusion holds in every
supercritical interval, not only below the first critical value.  Third, the
critical set depends only on the positive region and on singular points lying
there.  The singular sources in $\Gamma^-$ and on $\Gamma$ do not contribute
to \eqref{eq:Lambda}.

\section{Proof of the main theorem}

\begin{proof}[Proof of Theorem~\ref{thm:main}]
The proof has three steps.

\smallskip
\noindent
\emph{Step 1: topology of the positive region.}
Because $M$ is a closed Riemann surface and $\chi(M)\leq0$, its genus is at
least one.  By \eqref{eq:h-hyp}, each component of $\Gamma$ is an
embedded $C^{2,\vartheta}$ circle, with $h$ of opposite signs on its two
sides.  Consequently, the closure of each negative component is a compact
embedded surface with boundary.  Since its interior is a disk, the
classification of compact surfaces with boundary implies that this closure
is a closed disk with one boundary component.  The same local sign
description shows that distinct negative closures are disjoint.

Removing finitely many pairwise disjoint open disks from a connected closed
surface leaves a connected compact surface with boundary and with the same
genus as the original surface.  Its interior is precisely $\Gamma^+$.  It
follows that $\Gamma^+$ is connected and is not simply connected.  In
particular, for every integer $k\geq1$, the topological hypothesis used in the
scalar supercritical existence theorem of \cite{DMLR} is satisfied.

\smallskip
\noindent
\emph{Step 2: solution of the scalar problem.}
Consider
\begin{equation}\label{eq:scalar-original}
 \Delta_g u+\lambda
 \left(
   \frac{he^u}{\displaystyle\int_Mhe^u\dd V_g}-1
 \right)
 =\sum_{\ell=1}^L4\pi\alpha_\ell(\delta_{p_\ell}-1).
\end{equation}
If $0<\lambda<8\pi$, boundedness of $H$ and the Moser--Trudinger
inequality give, for $w\in\mathcal X_H$,
\begin{equation}\label{eq:coercivity}
 \mathcal J_\lambda(w)\geq
 \left(\frac12-\frac{\lambda}{16\pi}\right)
 \int_M|\nabla_gw|^2\dd V_g-C.
\end{equation}
Let $(w_m)$ be a minimizing sequence.  By \eqref{eq:coercivity} and
Poincar\'e's inequality it is bounded in $H^1(M)$.  After passing to a
subsequence, it converges weakly in $H^1(M)$ and almost everywhere to a
mean-zero function $w$.  Moser--Trudinger also bounds $e^{w_m}$ in
$L^p(M)$ for every fixed finite $p>1$, so uniform integrability yields
$e^{w_m}\to e^w$ in $L^1(M)$.  Therefore
$\int_MHe^{w_m}\dd V_g\to\int_MHe^w\dd V_g$.
Since the energies are bounded above and the Dirichlet term is nonnegative,
these integrals are bounded away from zero.  Hence $w\in\mathcal X_H$,
and weak lower semicontinuity proves that $w$ is a minimizer.  Its
Euler--Lagrange equation is \eqref{eq:regularized-scalar}.  Applying
\eqref{eq:normalization-shift} and setting
\begin{equation}\label{eq:recover-u}
 u=v-4\pi\sum_{\ell=1}^L\alpha_\ell G(\cdot,p_\ell)
\end{equation}
gives a solution of \eqref{eq:scalar-original}, with
$\int_Mhe^u\dd V_g=1$.

Suppose next that $\lambda>8\pi$ and
$\lambda\notin\Lambda_h$.  Since every positive integer multiple of $8\pi$
belongs to \eqref{eq:Lambda}, there is an integer $k\geq1$ such that
\eqref{eq:supercritical-interval} holds.  By Step~1, $\Gamma^+$ has a
non-simply-connected component.  The existence theorem of De
Marchis--L\'opez-Soriano--Ruiz \cite[Theorem~2.2]{DMLR} therefore produces a
critical point of \eqref{eq:scalar-functional} when the singular points avoid
$\Gamma$, after deleting any zero-strength marked points.
For positive singular sources on $\Gamma$, the nodal a priori estimate
of \cite{Wu} supplies the extension of the compactness input of
\cite[Theorem~2.1]{DMLR}.  The same scalar variational construction then
applies under \eqref{eq:source-hyp}, as explained in \cite{Wu}.
Normalize the resulting critical point by \eqref{eq:normalization-shift}
and recover $u$ using \eqref{eq:recover-u}.  Thus
\eqref{eq:scalar-original} has a solution for every $\lambda$ satisfying
\eqref{eq:lambda-noncritical}.

For completeness, the role of the exclusion
$\lambda\notin\Lambda_h$ is the following.  Compactness for the approximating
Euler--Lagrange equations supplies the deformation lemma used in the scalar
min--max construction.  If a normalized sequence of scalar solutions blows
up with parameters converging to $\lambda_*>0$, quantization forces
$\lambda_*\in\Lambda_h$; nodal and negative-region singular points do not
contribute.  This is a necessary condition for blow-up, not an assertion that
every element of $\Lambda_h$ is realized by a blowing-up sequence.
%No coefficient homotopy is involved.

\smallskip
\noindent
\emph{Step 3: lifting the scalar solution.}
Let $u$ be a solution of \eqref{eq:scalar-original}.  From
\eqref{eq:q-def}--\eqref{eq:rho-def},
\begin{equation}\label{eq:Arho}
 A\rho=\lambda A q=\lambda\one,
\end{equation}
and hence
\begin{equation}\label{eq:row-identity}
 \sum_{j=1}^na_{ij}\rho_j=\lambda,
 \qquad i=1,\ldots,n.
\end{equation}
Set $u_i=u$ for every $i$.  Then
\[
 \frac{he^{u_j}}{\int_Mhe^{u_j}\dd V_g}
 =
 \frac{he^u}{\int_Mhe^u\dd V_g}
 \qquad\text{for every }j.
\]
Using \eqref{eq:row-identity}, the left-hand side of the $i$th equation of
\eqref{eq:system-intro} becomes
\begin{align*}
 \Delta_g u_i+
 \sum_{j=1}^na_{ij}\rho_j
 \left(
   \frac{he^{u_j}}{\int_Mhe^{u_j}\dd V_g}-1
 \right)
 &=\Delta_g u+\lambda
 \left(
   \frac{he^u}{\int_Mhe^u\dd V_g}-1
 \right)\\
 &=\sum_{\ell=1}^L4\pi\alpha_\ell(\delta_{p_\ell}-1).
\end{align*}
Thus $(u,\ldots,u)$ solves the full system.  Condition
\eqref{eq:positive-denominator} follows from the scalar construction.

Finally, if $(u,\ldots,u)$ and $(\widetilde u,\ldots,\widetilde u)$
are equivalent modulo independent componentwise constants, then
$u-\widetilde u$ is constant.  Hence the diagonal map is injective on
solution classes.  Equivalently, use the representatives satisfying
$\int_Mhe^u\dd V_g=1$.  Any applicable scalar multiplicity lower bound
therefore transfers unchanged.  This completes the proof.
\end{proof}

\section{Parameter geometry and scope}\label{sec:scope}

The synchronized ray has a simple algebraic description in parameter space.
In this section we retain assumptions \ref{A1}--\ref{A3}.  Define
\begin{equation}\label{eq:ratio}
 \mathscr R_A(\rho)
 =\frac{\rho^TA\rho}{\sum_{i=1}^n\rho_i}.
\end{equation}
If $\rho=\lambda q$, then by \eqref{eq:Arho},
\begin{equation}\label{eq:ratio-ray}
 \mathscr R_A(\lambda q)=\lambda.
\end{equation}
Thus the synchronized ray meets the level hypersurfaces associated with
the scalar exceptional values,
\[
 \rho^TA\rho=\sigma\sum_{i=1}^n\rho_i,
 \qquad \sigma\in\Lambda_h,
\]
at $\rho=\sigma q$.  This identity does not identify these hypersurfaces
as the compactness thresholds of the full system.  Such an identification
would require an independent analysis of nonsynchronized blow-up.

The regions between consecutive level sets of $\mathscr R_A$ are connected.
Indeed, writing $\rho=r\theta$ with $\theta_i>0$ and
$\sum_i\theta_i=1$ gives
\[
 \mathscr R_A(r\theta)=r\,\theta^TA\theta.
\]
By \ref{A1}, $\theta^TA\theta>0$ on the open positive simplex.
For any $0\leq a<b$, the map
$(\theta,t)\mapsto t\theta/(\theta^TA\theta)$ is a homeomorphism from
that simplex times $(a,b)$ onto
$\{\rho\in(0,\infty)^n:a<\mathscr R_A(\rho)<b\}$.
This proves the asserted connectedness.
This observation suggests a possible extension of Theorem~\ref{thm:main}: if
one could compute a nonzero full system degree at a synchronized parameter and
prove compactness while varying $\rho$ inside one such region, then existence
would propagate throughout that connected region.  Neither compactness
nor a degree calculation for the full system follows from scalar existence.

The present argument does not supply that degree computation.  The system
functional on the product mean-zero domain $\mathcal X_H^n$ is
\begin{equation}\label{eq:system-functional}
 \mathcal J_\rho(\boldsymbol v)
 =\frac12\sum_{i,j=1}^na^{ij}
   \int_M\nabla_gv_i\cdot\nabla_gv_j\dd V_g
  -\sum_{i=1}^n\rho_i\log\int_MHe^{v_i}\dd V_g.
\end{equation}
Under \ref{A1}--\ref{A3}, the quadratic form of $A^{-1}$ is indefinite:
its trace is nonpositive by \ref{A2}, whereas
$\one^TA^{-1}\one=\sum_iq_i>0$ by \ref{A3}, so it has both a positive
and a negative eigenvalue.  Consequently, low values of
\eqref{eq:system-functional} may be produced by transverse oscillatory
directions rather than by concentration of positive measures, so the scalar
barycenter topology cannot simply be copied.  The diagonal restriction is
precisely what eliminates these transverse directions and leads to the exact
scalar reduction used above.  More explicitly,
\[
 \mathcal J_{\lambda q}(w\one)
 =\left(\sum_iq_i\right)\mathcal J_\lambda(w),
 \qquad w\in\mathcal X_H,
\]
and the passage from a scalar critical point to a full system solution is
justified by the componentwise identity \eqref{eq:row-identity}.

Theorem~\ref{thm:main} should therefore be understood as an intrinsic
existence theorem for a matrix-determined family of genuinely coupled
equations.  It provides solutions without crossing a degenerate coefficient
and identifies full-system compactness and a nonzero degree as additional
requirements for the proposed continuation to general parameter vectors.

\end{document}